\documentclass[12pt]{amsart}
\newtheorem{thm}{Theorem}[section]
\newtheorem{cor}[thm]{Corollary}
\newtheorem{lem}[thm]{Lemma}
\newtheorem{prop}[thm]{Proposition}
\theoremstyle{definition}
\newtheorem{defn}[thm]{Definition}
\theoremstyle{remark}
\newtheorem{rem}[thm]{Remark}
\numberwithin{equation}{section}
\def\cb{\mathcal{B}}

\def\bc{{\mathbb C}}

\def\bm{{\mathbb M}}

\def\br{{\mathbb R}}

\def\a{\alpha}
\def\b{\beta}
\def\g{\gamma}  \def\G{\Gamma}
  \def\D{\Delta}

\def\l{\lambda}

\def\p{\psi}

\def\s{\sigma} 
\def\t{\tau}
\def\f{\varphi}

\def\tr{\mathop{\rm Tr}}
\def\id{{\bf 1}\!\!{\rm I}}
\def\fb{{\mathbf{f}}}
\def\ab{{\mathbf{a}}}
\def\bb{{\mathbf{b}}}
\def\cb{{\mathbf{c}}}
\def\eb{{\mathbf{e}}}
\def\pb{{\mathbf{p}}}

\def\bx{{\mathbf{b}}}
\def\rb{{\mathbf{r}}}

\def\db{{\mathbf{d}}}
\def\gbb{{\mathbf{g}}}

\def\D{\Delta}
\def\G{\Gamma}

\def\wb{{\mathbf{w}}}

\def\o{\otimes}

\def\a{\alpha}
\begin{document}
\title[Pure Quasi quantum quadratic operators]
{On pure quasi quantum quadratic operators of $\bm_2(\mathbb{C})$}

\author{Farrukh Mukhamedov}
\address{Farrukh Mukhamedov\\
 Department of Computational \& Theoretical Sciences\\
Faculty of Science, International Islamic University Malaysia\\
P.O. Box, 141, 25710, Kuantan\\
Pahang, Malaysia} \email{{\tt far75m@yandex.ru}, {\tt
farrukh\_m@iium.edu.my}}

\author{Abduaziz Abduganiev}
\address{Abduaziz Abduganiev\\
 Department of Computational \& Theoretical Sciences\\
Faculty of Science, International Islamic University Malaysia\\
P.O. Box, 141, 25710, Kuantan\\
Pahang, Malaysia} \email{{\tt azizi85@yandex.ru}}
%\thanks{}%
\begin{abstract}

In the present paper we study quasi quantum quadratic operators
(q.q.o) acting on the algebra of $2\times 2$ matrices $\bm_2(\bc)$.
It is known that a channel is called pure if it sends pure states to
pure ones. In this papers, we introduce a weaker condition, called
$q$-purity, than purity of the channel. To study $q$-pure channels,
we concentrate ourselves to quasi q.q.o. acting on $\bm_2(\bc)$. We
describe all trace-preserving quasi q.q.o. on $\bm_2(\bc)$, which
allowed us to prove that if a trace-preserving symmetric quasi
q.q.o. such that the corresponding quadratic operator is linear,
then its $q$-purity implies its positivity. If a symmetric quasi
q.q.o. has a Haar state $\t$, then its corresponding quadratic
operator is nonlinear, and it is proved that such $q$-pure symmetric
quasi q.q.o. cannot be positive. We think that such a result will
allow to check whether a given mapping from $\bm_2(\bc)$ to
$\bm_2(\bc)\o\bm_2(\bc)$ is pure or not. On the other hand, our
study is related to construction of pure quantum nonlinear channels.
Moreover, it is also considered that nonlinear dynamics associated
with quasi pure q.q.o. may have differen kind of dynamics, i.e. it
may behave chaotically or trivially, respectively.

\end{abstract}

\maketitle
% ----------------------------------------------------------------

\section{Introduction}

It is known that entanglement is one of the essential features of
quantum physics and is fundamental in modern quantum technologies
\cite{N}. One of the central problems in the entanglement theory is
the discrimination between separable and entangled states. There are
several tools which can be used for this purpose. The most general
consists in applying the theory of linear positive maps \cite{Paul}.
In these studies one of the goal is to construct a map from the
state space of a system to the state space of another system. In the
literature on quantum information and communication systems, such a
map is called a channel \cite{N}. Note that the concept of state in
a physical system is a powerful weapon to study the dynamical
behavior of that system. One of the important class of channels is
so-called pure ones, which map pure states to pure ones (see
\cite{ACK,AO}). For example, important examples of such kind of maps
are conjugation of automorphisms of given algebra. But, if a channel
acts from algebra to another one, then the description of pure
channels is a tricky job. Therefore, it would be interesting
characterize such kind of maps (or channels). Note that quantum
mutual entropy of such kind of maps can be calculated easier way
than others \cite{OP,OV}.

In the present paper we are going to describe pure quasi quantum
quadratic operators (see also \cite{MATA}). On the other hand, such
kind of operators define quadratic operators. We should stress that
quadratic dynamical systems have been proved to be a rich source of
analysis for the investigation of dynamical properties and modeling
in different domains, such as population dynamics \cite{Be,FG,HJ},
physics \cite{PL,UR}, economy \cite{D}, mathematics
\cite{HS,L1,U,V}. The problem of studying the behavior of
trajectories of quadratic stochastic operators was stated in
\cite{U}. The limit behavior and ergodic properties of trajectories
of such operators were studied in \cite{K,L1,L2,Ma,V}. However, such
kind of operators do not cover the case of quantum systems.
Therefore, in \cite{GM1,GM2} quantum quadratic operators acting on a
von Neumann algebra were defined and studied. Certain ergodic
properties of such operators were studied in \cite{M2,M3}. In those
papers basically dynamics of quadratic operators were defined
according to some recurrent rule (an analog of Kolmogorov-Chapman
equation) which makes a possibility to study asymptotic behaviors of
such operators. However, with a given quadratic operator one can
define also a non-linear operator whose dynamics (in non-commutative
setting) is not well studied yet. Some class of such kind of
operators defined on $M_2(\bc)$ has been studied in \cite{MA,MATA}.
Note that in \cite{MM0} another construction of nonlinear quantum
maps were suggested and some physical explanations of such nonlinear
quantum dynamics were discussed. In all these investigations, the
said quantum quadratic operators by definition are positive. But, in
general, to study the nonlinear dynamics the positivity of the
operator is strong condition. Therefore, in the present paper we are
going to introduce a weaker than the positivity, and corresponding
operators are called quasi quantum quadratic. In the paper we
concentrate ourselves to trace-preserving operators acting on
$\bm_2(\bc)$. Each such kind of operator defines a quadratic
operator acting on state space of $\bm_2(\bc)$. It is known that a
mapping is called {\it pure} if it sends pure states to pure ones.
In this papers, we introduce a weaker condition, called $q$-purity,
than purity of the mapping. To study $q$-pure channels, we
concentrate ourselves to quasi q.q.o. acting on $\bm_2(\bc)$.  We
call such an operator $q$-pure, if its corresponding quadratic
operator maps pure state to pure ones. We first describe all
trace-preserving quasi q.q.o. on $\bm_2(\bc)$, which allowed us to
describe all $q$-pure quadratic operators. Then we prove that if a
trace-preserving symmetric quasi q.q.o. such that the corresponding
quadratic operator is linear, then its $q$-purity implies its
positivity. Moreover, if a symmetric quasi q.q.o. has a Haar state
$\t$, then its corresponding quadratic operator is nonlinear, and it
is proved that such $q$-pure symmetric quasi q.q.o. cannot be
positive. We think that such a result will allow to check whether a
given mapping from $\bm_2(\bc)$ to $\bm_2(\bc)\o\bm_2(\bc)$ is pure
or not. On the other hand, our study is related to construction of
pure quantum nonlinear channels. Besides, it is also considered that
nonlinear dynamics associated with quasi pure q.q.o. may have
differen kind of dynamics, i.e. it may behave chaotically or
trivially, respectively.

\section{Preliminaries}

Let $B(H)$ be the set of linear bounded operators from a complex
Hilbert space $H$ to itself. By $B(H)\o B(H)$ we mean tensor
product of $B(H)$ into itself. In the sequel $\id$ means an
identity matrix. By $B(H)^*$ it is usually denoted the conjugate
space of $B(H)$. We recall that a linear functional $\f\in B(H)^*$
is called {\it positive} if $\f(x)\geq 0$ whenever $x\geq 0$. The
set of all positive linear functionals is denoted by $B(H)^*_+$. A
positive functional $\f$ is called {\it state} if $\f(\id)=1$. By
$S(B(H))$ we denote the set of all states defined on $B(H)$.

Let $\D:B(H)\to B(H)\o B(H)$ be a linear operator. Then $\D$
defines a conjugate operator $\D^*:(B(H)\o B(H))^* \to B(H)^*$ by
\begin{eqnarray*}
  \D^*(f)(x)=f(\D x), \ f\in(B(H)\o B(H))^*, \
  x\in B(H).
\end{eqnarray*}
One can define an operator $V_{\D}$ by
\begin{eqnarray*}
  V_{\D}(\varphi)=\D^*(\f\o\f), \ \f\in B(H)^*.
\end{eqnarray*}

 Let $U:B(H)\o B(H)\to B(H)\o B(H)$
be a linear operator such that $U(x\o y)=y\o x$ for all $x,y\in
\bm_2(\bc)$.

\begin{defn}\label{qqso} A linear operator $\D:B(H)\to B(H)\o B(H)$ is said to be
\begin{enumerate}
    \item[(a)] -- a {\it quasi quantum quadratic operator (quasi q.q.o)}  if it is
unital (i.e. $\D\id=\id\o\id$), *-preserving (i.e.
$\D(x^*)=\D(x)^*, \ \forall x\in B(H)$) and
\begin{eqnarray*}
  V_{\D}(\f)\in B(H)^*_+ \ \ \textrm{whenever} \
  \f\in B(H)^*_+;
\end{eqnarray*}
    \item[(b)] -- a {\it quantum quadratic
operator (q.q.o.)} if it is unital (i.e. $\D\id=\id\o\id$) and
positive ( i.e. $\D x\geq 0$ whenever $x\geq 0$);
    \item[(c)] -- a
{\it quantum convolution} if it is a q.q.o. and satisfies
coassociativity condition:
$$(\D\o
id)\circ\D=(id\o\D)\circ\D,$$ where $id$ is the identity operator
of $\bm_2(\bc)$;
\item[(d)] -- a {\it symmetric} if one has
$U\D=\D$.
\end{enumerate}
\end{defn}

One can see that if $\D$ is q.q.o. then it is a quasi q.q.o. A
state $h\in S(B(H))$ is called {\it a Haar state} for a quasi
q.q.o. $\D$ if for every $x\in B(H)$ one has
\begin{equation}\label{Haar}
(h\o id)\circ \D(x)=(id\o h)\circ\D(x)=h(x)\id.
\end{equation}

\begin{rem}\label{qg} Note that if a quantum convolution $\D$ on
$B(H)$ becomes a $*$-homomorphic map with a condition
$$
\overline{\textrm{Lin}}((\id\o
B(H))\D(B(H)))=\overline{\textrm{Lin}}((B(H)\o\id)\D(B(H)))=B(H)\o
B(H)
$$
then a pair $(B(H),\D)$ is called a {\it compact quantum group}
\cite{W,S}. It is known \cite{W} that for given any compact
quantum group there exists a unique Haar state w.r.t. $\D$.
\end{rem}

\begin{rem} In \cite{M2} it has been studied symmetric q.q.o., which was
called {\it quantum quadratic stochastic operator}.
\end{rem}

\begin{rem} We note that there is another approach to nonlinear quantum operators on $C^*$-algebras (see \cite{MM0}).
\end{rem}

Note that from unitality of $\Delta$ we conclude that for any
quasi q.q.o. $V_\D$ maps $S(B(H))$ into itself. In some literature
operator $V_\D$ is called {\it quadratic convolution} (see for
example \cite{FS}). In \cite{MATA} certain dynamical properties of
$V_\Delta$ associated with q.q.o. defined on $\bm_2(\bc)$ are
investigated. In \cite{MA} Kadison-Schwarz property of q.q.o. has
been studied.

In quantum information, pure channels play important role, which can
be defined as follows: a channel (i.e. positive and unital mapping)
$T:B(H_1)\to B(H_2)$ is called {\it pure} if for any pure state
$\f\in S(B(H_1))$ the state $T^*\f$ is also pure (see \cite{AO}). It
is essential to describe such channels. Of course, if $H_1=H_2$ then
one can see that automorphisms of $B(H_1)$ are examples of pure
channels. But, in general, the description of pure channels is a
tricky job.

Now let us assume that $\D$ be a pure q.q.o. Then for any pure
states $\f,\p\in S(B(H))$ one concludes that $\D^*(\f\o\p)$ is also
pure. In particularly, for any pure $\f\in S(B(H)$ we have
$\D^*(\f\o\f)$ is also pure. Note that the reverse is not true.
Therefore, in this paper we are going to define more weaker notion
than purity for quasi q.q.o.

\begin{defn}
  A quasi q.q.o. $\D$ is called {\it q-pure} if for any pure state $\f$ the state
  $V_{\D}(\f)$ is also pure.
\end{defn}

From this definition one can immediately see that purity of quasi
q.q.o. implies its $q$-purity.

\section{Quasi quantum quadratic operators on $\bm_2(\bc)$}

By $\bm_2(\bc)$ be an algebra of $2\times 2$ matrices over complex
field $\bc$. In this section we are going to describe quantum
quadratic operators on $\bm_2(\bc)$ as well as find necessary
conditions for such operators to satisfy the Kadison-Schwarz
property.

Recall \cite{BR} that the identity and Pauli matrices $\{ \id,
\sigma_1, \sigma_2, \sigma_3 \}$ form a basis for $\bm_2(\bc)$,
where
\begin{eqnarray*}
\sigma_1 = \left( \begin{array}{cc} 0 & 1 \\ 1 & 0 \end{array}
\right)~~ \sigma_2 = \left( \begin{array}{cc} 0 & -i \\ i & 0
\end{array} \right)~~ \sigma_3 = \left( \begin{array}{cc} 1 & 0 \\
0 & -1 \end{array} \right).
\end{eqnarray*}

In this basis every matrix $x\in\bm_2(\bc)$ can  be written as $x =
w_0\id + \wb{\bf \sigma}$ with $w_0\in\bc$, $\wb =(w_1,w_2,w_3)\in
\bc^3$, here $\wb\s=w_1\s_1+w_2\s_2+w_3\s_3$. In what follows, we
frequently use notation
$\overline{\wb}=(\overline{w_1},\overline{w_2},\overline{w_3})$.

\begin{lem}\label{m2}\cite{RSW} The following assertions hold true:
\begin{enumerate}
\item[(a)] $x$ is self-adjoint iff  $w_0,\wb$  are reals;
\item[(b)]$\tr(x) = 1$ iff $w_0 =0.5$, here $\tr$ is the trace of
a matrix $x$;
\item[(c)] $x> 0$ iff $\|\wb\|\leq w_0$, where
$\|\wb\|=\sqrt{|w_1|^2+|w_2|^2+|w_3|^2}$;
\item[(d)] A linear
functional $\f$ on $\bm_2(\bc)$ is a state iff it can be
represented by
\begin{equation}\label{state}
{\f}(w_0\id + \wb\sigma)=w_0+\langle\wb,{\mathbf{f}}\rangle, \ \
\end{equation}
where ${\mathbf{f}}=(f_1,f_2,f_3)\in\br^3$ such that
$\|{\mathbf{f}}\|\leq 1$. Here as before $\langle\cdot,\cdot\rangle$
stands for the scalar product in $\bc^3$. \item[(e)] A state $\f$ is
a pure if and only if $\|\fb\|=1$. So pure states can be seen as the
elements of unit sphere in $\br^3$.
\end{enumerate}
\end{lem}

In the sequel we shall identify a state with a vector $\fb\in
\br^3$. By $\t$ we denote a normalized trace, i.e.
$\t(x)=\frac{1}{2}\tr(x)$, $x\in \bm_2(\bc)$.

Let  $\Delta:\bm_2(\bc)\rightarrow \bm_2(\bc) \otimes \bm_2(\bc)$ be
a quasi q.q.o. Then we write the operator $\Delta $ in terms of a
basis in $\bm_2(\bc)\o\bm_2(\bc)$ formed by the Pauli matrices.
Namely,
\begin{eqnarray}\label{D-bij}
&& \Delta \id=\id\otimes \id; \nonumber\\
&& \Delta (\sigma_i)=b_i(\id\otimes \id)+\overset{3}{\underset{j=1}{\sum
    }}b_{ji}^{(1)}(\id\otimes \sigma_j)+\overset{3}{\underset{j=1}{\sum
    }}b_{ji}^{(2)}(\sigma_j \otimes \id)+\overset{3}{\underset{m,l=1}{\sum
    }}b_{ml,i}(\sigma_m \otimes \sigma_l),
\end{eqnarray}
where $i=1,2,3$.

In general, a description of positive operators is one of the main
problems of quantum information. In the literature most tractable
maps are positive and trace-preserving ones, since such maps arise
naturally in quantum information theory (see \cite{N}).  Therefore,
in the sequel we shall restrict ourselves to trace-preserving quasi
q.q.o., i.e. $\t\o\t\circ\D=\t$. So, we would like to describe all
such kind of maps.

\begin{prop}\label{trace-pre}
Let $\D:\bm_2(\bc)\to \bm_2(\bc)\o\bm_2(\bc) $ be a trace-preserving
quasi q.q.o., then in \eqref{D-bij} one has $b_j=0$, and
$b^{(1)}_{ij}$, $b^{(2)}_{ij}$, $b_{ij,k}$ are real for every
$i,j,k\in\{1,2,3\}$. Moreover, $\D$ has the following form:
\begin{equation}\label{D3}
\D(x)=w_0\id\otimes\id+
\id\otimes\mathbf{B}^{(1)}\wb\cdot\s+\mathbf{B}^{(2)}\wb\cdot\s\otimes\id+
\sum_{m,l=1}^3\langle\bx_{ml},\overline{\wb}\rangle\sigma_m\otimes\sigma_l,
\end{equation}
where $x=w_0\id+\wb\s$, $\bx_{ml}=(b_{ml,1},b_{ml,2},b_{ml,3})$, and
$\mathbf{B}^{(k)}=(b_{ij}^{(k)})_{i,j=1}^3$, $k=1,2$. Here as before
$\langle\cdot,\cdot\rangle $ stands for the standard scalar product
in $\bc^3$.
\end{prop}

\begin{proof}
From the *-preserving condition we get
\begin{eqnarray*}
\Delta (\sigma_i^{*})&=&\overline{b_i}(\id\otimes
\id)+\overset{3}{\underset{j=1}{\sum
    }}\overline{b_{ji}^{(1)}}(\id\otimes \sigma_j)+\overset{3}{\underset{j=1}{\sum
    }}\overline{b_{ji}^{(2)}}(\sigma_j \otimes \id)+\overset{3}{\underset{m,l=1}{\sum
    }}\overline{b_{ml,i}}(\sigma_m \otimes \sigma_l).
\end{eqnarray*}
This yields that $b_i=\overline{b_i}$,
$b_{ji}^{(k)}=\overline{b_{ji}^{(k)}}$ ($k=1,2$) and
$b_{ml,i}=\overline{b_{ml,i}}$, i.e. all coefficients are real
numbers.

Using the trace-preserving condition one finds
\begin{eqnarray*}
\tau\o\tau(\D(\s_i))=b_i=\tau(\s_i).
\end{eqnarray*}
Therefore, $b_i=0, \ i=1,2,3.$ Hence, $\D$ has the following form
\begin{equation}\label{d11}
\Delta(\sigma_i)=\overset{3}{\underset{j=1}{\sum
    }}b_{ji}^{(1)}(\id\otimes \sigma_j)+\overset{3}{\underset{j=1}{\sum
    }}b_{ji}^{(2)}(\sigma_j \otimes \id)+\overset{3}{\underset{m,l=1}{\sum
    }}b_{ml,i}(\sigma_m \otimes \sigma_l), \ \ i=1,2,3.
\end{equation}

Denoting
\begin{equation}\label{b-ml}
\mathbf{B}^{(k)}=(b_{ij}^{(k)})_{i,j=1}^3,\ k=1,2, \ \ \ \ \
\bx_{ml}=(b_{ml,1},b_{ml,2},b_{ml,3})
\end{equation}
and taking any $x=w_0\id+\wb\sigma\in \bm_2(\bc)$, from \eqref{d11}
we immediately find \eqref{D3}. This completes the proof.
\end{proof}

One can rewrite \eqref{D3} as follows
\begin{equation}\label{D3-0}
\D(x)=\l\D_1(x)+(1-\l)\D_2(x),
\end{equation}
where
\begin{eqnarray}\label{D3-1}
&& \D_1(x)=w_0\id\otimes\id+
\frac{1}{\l}\sum_{m,l=1}^3\langle\bx_{ml},\overline{\wb}\rangle\sigma_m\otimes\sigma_l,
\\[2mm]
\label{D3-2} && \D_2(x)=w_0\id\otimes\id+
\frac{1}{1-\l}\bigg(\mathbf{B}^{(2)}\wb\cdot\s\otimes\id+\id\otimes\mathbf{B}^{(1)}\wb\cdot\s\bigg).
\end{eqnarray}

Now assume that $b_{ij,k}=0$ for all $i,j,k\in\{1,2,3\}$ and $\D$ is
$q$-pure symmetric quasi q.q.o. In this case, $\D$ has the following
form
\begin{eqnarray}\label{D31-2}
\D(w_0\id+\wb\s)=w_0\id\otimes\id+\mathbf{B}\wb\cdot\s\otimes\id+\id\otimes\mathbf{B}\wb\cdot\s.
\end{eqnarray}
Let us take any $\f\in S(\bm_{2}(\bc))$ and $\fb\in\br^3$ be the
corresponding vector. Then we find
\begin{equation*}
\f\o\f(\D(w_0\id+\wb\s))=w_0+2\langle
\mathbf{B}\wb,\fb\rangle=w_0+\langle\wb,2\mathbf{B}^*\fb\rangle
\end{equation*}
Hence, if $\f$ is pure, then $\|\fb\|=1$. Denoting
$\mathbf{U}=2\mathbf{B}^*$ and the $q$-purity of $\D$ yields that
$\|\mathbf{U}\fb\|=1$ for all $\fb$ with $\|\fb\|=1$. This means
that $\mathbf{U}$ is isometry, so $\|\mathbf{U}\|=1$, i.e.
$\|\mathbf{B}\|=1/2$. Consequently, one concludes that $\D$ is
$q$-pure if and only if $2\mathbf{B}$ is isometry.

Now we are interested, whether $q$-pure symmetric quasi q.q.o. is
positive. To answer to this question we need some auxiliary facts.

\begin{lem}\label{L1}
  Let
  $x=w_0\id\otimes\id+\wb\cdot\s\otimes\id+\id\otimes\rb\cdot\s$. Then
  the following statements hold true:
  \begin{enumerate}
    \item[(i)] $x$ is self-adjoint if and only if $w_0 \in
    \mathbb{R}$ and $\wb, \rb \in \mathbb{R}^3$;
    \item[(ii)] $x$ is positive if and only if $w_0>0$ and $\|\wb\|+\|\rb\|\leq w_0$.
  \end{enumerate}
\end{lem}
\begin{proof} (i). One can see that
    \begin{eqnarray*}
      x^* = \overline{w_0}
      \id\otimes\id+\overline{\wb}\cdot\s\otimes\id +
      \id\otimes\overline{\rb}\cdot\s
    \end{eqnarray*}
So, self adjointness $x$ implies $\overline{w_0}=w_0$,
    $\overline{\wb}=\wb$, $\overline{\rb}=\rb$.

(ii). Let $x$ be self-adjoint. Then from the definition of Pauli
matrices one finds
    \begin{eqnarray*}
      x=\left(
          \begin{array}{cccc}
            w_0+w_3+r_3 & w_1-iw_2 & r_1-ir_2 & 0 \\
            w_1+iw_2 & w_0-w_3+r_3 & 0 & r_1-ir_2 \\
            r_1+ir_2 & 0& w_0+w_3-r_3 & w_1-iw_2 \\
            0 & r_1+ir_2 & w_1+iw_2 & w_0-w_3-r_3 \\
          \end{array}
        \right)
    \end{eqnarray*}
It is easy to calculate that eigenvalues of last matrix are the
    followings
    \begin{eqnarray*}
&&      \l_1=w_0-\|\rb\|+\|\wb\|, \ \
      \l_2=w_0-\|\rb\|-\|\wb\|, \\[2mm]
  &&    \l_3=w_0+\|\rb\|+\|\wb\|, \ \
      \l_4=w_0+\|\rb\|-\|\wb\|
    \end{eqnarray*}

    So, we can conclude that $x$ is positive if and only if the smallest
    eigenvalue is positive. This means $w_0-\|\rb\|-\|\wb\|\geq0$,
    which completes the proof.
\end{proof}

\begin{prop}\label{Pos}
  The mapping $\D$ given by \eqref{D31-2} is positive if and only if
 $\|\mathbf{B}\|\leq 1/2$.
\end{prop}

\begin{proof}
  Let $x=w_0\id+\wb\cdot\s$ be positive, i.e. $w_0>0,$ $\|\wb\|\leq
  w_0$. Without lost of generality we may assume $w_0=1.$
  Now Lemma $\ref{L1}$ yields that $\D(x)$ is positive if and only if
  $2\|\mathbf{B}\wb\|\leq1$. This competes the
  proof.
\end{proof}

From this Proposition and above made conclusions we immediately get
the following

\begin{thm}\label{P-q} Let $\D$ be given by
\eqref{D31-2}. Then the following statements hold true:
\begin{enumerate}
\item[(i)] $\D$ is quasi q.q.o. if and only if $\D$ is positive, i.e. $\|\mathbf{B}\|\leq 1/2$;
\item[(ii)]  $\D$ is $q$-pure if and only if $2\mathbf{B}$ is isometry. Moreover,
$\D$ is positive.
\end{enumerate}
\end{thm}

Note that using the methods of \cite{MA2} one may study
Kadison-Schwarz property of mappings given by \eqref{D31-2}. Now the
question is what about the case when $b_{ij,k}\neq 0$. Therefore,
the next section is devoted to this this question.

\section{Q-pure symmetric quasi quantum quadratic operators on $\bm_2(\bc)$}

In this section we are going to describe trace-preserving $q$-pure
symmetric quasi q.q.o.

Denote
\begin{eqnarray*}
 &&{\mathbf{D}}=\{\pb=(p_1,p_2,p_3)\in\mathbb{R}:p_1^2+p_2^2+p_3^2\leq 1\}, \\
 &&{\mathbf{S}}=\{\pb=(p_1,p_2,p_3)\in\mathbb{R}:p_1^2+p_2^2+p_3^2=1\}.
\end{eqnarray*}

Let $\Delta$ be a trace-preserving symmetric quasi q.q.o. on
$\bm_2(\bc)$. Then due to Lemma \ref{m2} (d) and Proposition
\ref{trace-pre} the functional $\Delta^{*}(\varphi\otimes \psi)$ is
a state if and only if the vector
$$
\textbf{f}_{\Delta^{*}(\varphi,
\psi)}=\bigg(\sum_{j=1}^3b_{j1}\big(p_j+f_j\big)+
\overset{3}{\underset{i,j=1}{\sum}}b_{ij,1}f_ip_j,
\sum_{j=1}^3b_{j2}\big(p_j+f_j\big)+\overset{3}{\underset{i,j=1}{\sum}}b_{ij,2}f_ip_j,
\sum_{j=1}^3b_{j3}\big(p_j+f_j\big)+\overset{3}{\underset{i,j=1}{\sum}}b_{ij,3}f_ip_j\bigg).
$$
satisfies $\|\fb_{\D^*(\f,\p)}\|\leq 1$.

Let us consider the quadratic operator, which is defined by
$V_\D(\f)=\D^*(\f\o\f)$, $\f\in S(\bm_{2}(\bc))$. From the last
expression we find that
\begin{eqnarray*}
V_\Delta(\varphi)(\s_k)=\sum_{j=1}^32b_{jk}f_j+\overset{3}{\underset{i,j=1}{\sum}}
b_{ij,k}f_if_j, \ \ \fb\in {\mathbf{D}}.
\end{eqnarray*}

This suggests us the consideration of a nonlinear operator
$V:{\mathbf{D}}\to {\mathbf{D}}$ defined by
\begin{equation}\label{V}
V(\textbf{f})_{k}=\sum_{j=1}^32b_{jk}f_j+\overset{3}{\underset{i,j=1}{\sum}}b_{ij,k}f_{i}f_{j},
  \ \ \  k=1,2,3.
\end{equation}
where $\fb=(f_1,f_2,f_3)\in \mathbf{D}$.

From the definition and Lemma \ref{m2} (e) we conclude that the
$\Delta$ is $q$-pure if and only if $V({\mathbf{S}})\subset
{\mathbf{S}}$.

{\sc Example.} Let us consider an example of pure symmetric quasi
q.q.o. Let
\begin{eqnarray*}
   \D_0(x)&=&w_0\id\o\id+w_1\s_1\o\s_2+w_1\s_2\o\s_1+w_2\s_1\o\s_1\\
   &&-w_2\s_2\o\s_2-w_2\s_3\o\s_3+w_3\s_1\o\s_3+
   w_3\s_3\o\s_1
   \end{eqnarray*}

Then the corresponding quadratic operator has the following form
\begin{eqnarray}\label{ex1}
  V_0(\fb)=\left\{\begin{array}{lll}
    2f_1f_2 \\
    f_1^2-f_2^2-f_3^2 \\
    2f_1f_3 \\
  \end{array}
  \right.
\end{eqnarray}
Let us show $V_0$ maps $\mathbf{S}$ to $\mathbf{S}$. Indeed, let
$\fb=(f_1,f_2,f_3)\in {\mathbf{S}}$, i.e. $f_1^2+f_2^2+f_3^2=1$.
Then we have
\begin{eqnarray*}
(2f_1f_2)^2+(f_1^2-f_2^2-f_3^2)^2+(2f_1f_3)^2&=&4f_1^2f_2^2+(2f_1^2-1)^2+4f_1^2f_3^2 \\
&=&4f_1^2(1-f_1^2-f_3^2)+4f_1^4-4f_1^2+1+4f_1^2f_3^2\\
&=&1
\end{eqnarray*}
This shows that $\Delta_0 $ is q-pure.\\

Now let us rewrite the quadratic operator $V$ (see \eqref{V}) as
follows
\begin{eqnarray}\label{qqqo1}
  V(\fb)=\left\{\begin{array}{ccc}
    a_1f_1^2+b_1f_2^2+c_1f_3^2+\a_1 f_1f_2+\b_1f_2f_3+\g_1f_1f_3+d_1f_1+e_1f_2+g_1f_3
    \\[2mm]
    a_2f_1^2+b_2f_2^2+c_2f_3^2+\a_2 f_1f_2+\b_2f_2f_3+\g_2f_1f_3+d_2f_1+e_2f_2+g_2f_3
    \\[2mm]
    a_3f_1^2+b_3f_2^2+c_3f_3^2+\a_3 f_1f_2+\b_3f_2f_3+\g_3f_1f_3+d_3f_1+e_3f_2+g_3f_3
  \end{array}
  \right.
\end{eqnarray}
where $\fb\in \mathbf{D}$.

\begin{thm}\label{qqqo2}
  The operator $V$ given by \eqref{qqqo1} maps ${\mathbf{S}}$ into itself if
  and only if  the followings hold true
  \begin{enumerate}
\item[(i)] $\|\ab\|^2+\|\db\|^2=1, \ \|\bb\|^2+\|\eb\|^2=1, \
\|\cb\|^2+\|\gbb\|^2=1;$

\item[(ii)]$\|A\|=\|\ab-\bb\|, \ \|\Gamma\|=\|\ab-\cb\|, \
\|B\|=\|\bb-\cb\|;$

\item[(iii)] $\langle\ab,\db\rangle=0, \ \langle\bb,\eb\rangle=0,
\ \langle\cb,\gbb\rangle=0;$

\item[(iv)] $\langle\ab,\G\rangle=\langle\cb,\G\rangle, \
\langle\bb,B\rangle=\langle\cb,B\rangle, \
\langle\ab,A\rangle=\langle\bb,A\rangle; $

\item[(v)] $\langle\cb,\G\rangle+\langle\db,\gbb\rangle=0, \
\langle\cb,B\rangle+\langle\eb,\gbb\rangle=0, \
\langle\cb,\db\rangle+\langle\G,\gbb\rangle=0, \\
\langle\cb,\eb\rangle+\langle B,\gbb\rangle=0, \
\langle\bb,\db\rangle+\langle A,\eb\rangle=0, \
\langle\bb,A\rangle+\langle\db,\eb\rangle=0, \\
\langle\bb,\gbb\rangle+\langle B,\eb\rangle=0, \
\langle\ab,\eb\rangle+\langle A,\db\rangle=0, \
\langle\ab,\gbb\rangle+\langle\G,\db\rangle=0;$

\item[(vi)] $\langle\ab,B\rangle-\langle\cb,B\rangle+\langle
A,\G\rangle=0, \ \langle\bb,\G\rangle-\langle\cb,\G\rangle+\langle
A,B\rangle=0, \\
\langle A,\gbb\rangle+\langle B,\db\rangle+\langle\G,\eb\rangle=0, \
\langle\cb,A\rangle+\langle\db,\eb\rangle+\langle B,\G\rangle=0,$
   \end{enumerate}
where $\ab=(a_1,a_2,a_3), \ \bb=(b_1,b_2,b_3), \ \cb=(c_1,c_2,c_3),
\ \db=(d_1,d_2,d_3), \ \eb=(e_1,e_2,e_3), \ \gbb=(g_1,g_2,g_3), \
\G=(\g_1,\g_2,\g_3), \ A=(\a_1,\a_2,\a_3), \ B=(\b_1,\b_2,\b_3).$
 \end{thm}

 \begin{proof}
   "only if" part. It is enough to show
   \begin{eqnarray}\label{qqqo3}
     \big(V(\fb)_1\big)^2+\big(V(\fb)_2\big)^2+\big(V(\fb)_3\big)^2=1
   \end{eqnarray}
   for any $\fb=(f_1,f_2,f_3)$ with $f_1^2+f_2^2+f_3^2=1.$

Let us rewrite \eqref{qqqo1} as follows
\begin{eqnarray}\label{qqqo4}
  && V(\fb)=\left\{\begin{array}{ccc}
    (a_1-c_1)f_1^2+(b_1-c_1)f_2^2+c_1+\a_1 f_1f_2+\b_1f_2f_3+\g_1f_1f_3+d_1f_1+e_1f_2+g_1f_3
    \\[2mm]
    (a_2-c_2)f_1^2+(b_2-c_2)f_2^2+c_2+\a_2 f_1f_2+\b_2f_2f_3+\g_2f_1f_3+d_2f_1+e_2f_2+g_2f_3
    \\[2mm]
    (a_3-c_3)f_1^2+(b_3-c_3)f_2^2+c_3+\a_3 f_1f_2+\b_3f_2f_3+\g_3f_1f_3+d_3f_1+e_3f_2+g_3f_3
    \\[2mm]
  \end{array}
  \right.
\end{eqnarray}
From \eqref{qqqo3}, \eqref{qqqo4} we derive
\begin{eqnarray*}
  &&\big((a_1-c_1)f_1^2+(b_1-c_1)f_2^2+c_1+\a_1
  f_1f_2+\b_1f_2f_3+\g_1f_1f_3+d_1f_1+e_1f_2+g_1f_3\big)^2\\[2mm]
  +&&\big((a_2-c_2)f_1^2+(b_2-c_2)f_2^2+c_2+\a_2
  f_1f_2+\b_2f_2f_3+\g_2f_1f_3+d_2f_1+e_2f_2+g_2f_3\big)^2\\[2mm]
  +&&\big((a_3-c_3)f_1^2+(b_3-c_3)f_2^2+c_3+\a_3
  f_1f_2+\b_3f_2f_3+\g_3f_1f_3+d_3f_1+e_3f_2+g_3f_3\big)^2=1
\end{eqnarray*}
After some calculations we obtain the following
\begin{eqnarray*}
  &&\big(\|\ab\|^2+\|\cb\|^2-\|\G\|^2-2\langle\ab,\cb\rangle\big)f_1^4+\big(\|\bb\|^2+\|\cb\|^2-\|B\|^2-2\langle\bb,\cb\rangle\big)f_2^4\\[2mm]
  +&&\big(2\langle\ab,A\rangle-2\langle
  B,\G\rangle-2\langle\cb,A\rangle\big)f_1^3f_2+\big(2\langle\ab,\G\rangle-2\langle\cb,\G\rangle\big)f_1^3f_3\\[2mm]
  +&&\big(2\langle\ab,\db\rangle-2\langle\cb,\db\rangle-2\langle\G,\gbb\rangle\big)f_1^3+\big(2\langle\bb,A\rangle-2\langle
  B,\G\rangle-2\langle\cb,A\rangle\big)f_1f_2^3\\[2mm]
  +&&\big(2\langle\bb,B\rangle-2\langle\cb,B\rangle\big)f_2^3f_3+\big(2\langle\bb,\eb\rangle-2\langle\cb,\eb\rangle-2\langle B,\gbb\rangle\big)f_2^3\\
  +&&\big(2\|\cb\|^2+\|A\|^2-\|B\|^2-\|\G\|^2+2\langle\ab,\bb\rangle-2\langle\bb,\cb\rangle-2\langle\ab,\cb\rangle\big)f_1^2f_2^2\\[2mm]
  +&&\big(2\langle\ab,B\rangle+2\langle
  A,\G\rangle-2\langle\cb,B\rangle\big)f_1^2f_2f_3+\big(2\langle\ab,\eb\rangle+2\langle A,\db\rangle-2\langle\cb,\eb\rangle-2\langle
  B,\gbb\rangle\big)f_1^2f_2\\[2mm]
  +&&\big(2\langle\ab,\gbb\rangle+2\langle\G,\db\rangle-2\langle\cb,\gbb\rangle\big)f_1^2f_3+\big(\|\G\|^2+\|\db\|^2-2\|\cb\|^2-\|\gbb\|^2+2\langle\ab,\cb\rangle\big)f_1^2\\[2mm]
  +&&\big(2\langle\bb,\G\rangle+2\langle
  A,B\rangle-2\langle\cb,\G\rangle\big)f_1f_2^2f_3+\big(2\langle\bb,\db\rangle+2\langle
  A,\eb\rangle-2\langle\cb,\db\rangle-2\langle\G,\gbb\rangle\big)f_1f_2^2\\[2mm]
  +&&\big(2\langle\bb,\gbb\rangle+2\langle
  B,\eb\rangle-2\langle\cb,\gbb\rangle\big)f_2^2f_3+\big(\|B\|^2+\|\eb\|^2-2\|\cb\|^2-\|\gbb\|^2+2\langle\bb,\cb\rangle\big)f_2^2\\[2mm]
  +&&\big(2\langle A,\gbb\rangle+2\langle
  B,\db\rangle+2\langle\G,\eb\rangle\big)f_1f_2f_3+\big(2\langle\cb,A\rangle+2\langle
  B,\G\rangle+2\langle\db,\eb\rangle\big)f_1f_2\\[2mm]
  +&&\big(2\langle\cb,\G\rangle+2\langle\db,\gbb\rangle\big)
  f_1f_3+\big(2\langle\cb,B\rangle+2\langle\eb,\gbb\rangle\big)
  f_2f_3+\big(2\langle\cb,\db\rangle+2\langle\G,\gbb\rangle\big)f_1\\[2mm]
  +&&\big(2\langle\cb,\eb\rangle+2\langle B,\gbb\rangle\big)f_2+2\langle\cb,\gbb\rangle f_3+\|\cb\|^2+\|\gbb\|^2-1=0
\end{eqnarray*}
which is satisfied (i)--(vi).

  "if" part is obvious. This completes the proof.
 \end{proof}

In what follows, we are interested in the case when $\D_2=0$ in
\eqref{D3-0}. This means that $\D$ has a Haar state $\t$. Indeed,
using the equality \eqref{Haar} with $h=\t$ one gets
$$
(id\o\t)(\D(\s_i))=\sum_{j=1}^3b_{ji}\s_j=\t(\s_i)\id=0, \ \
i=1,2,3.
$$
Therefore, $b_{ji}=0$ for all $i,j\in\{1,2,3\}$. Hence, $\D$ has the
following form
\begin{equation}\label{D4-1}
\D(w_0\id+\wb\s)=w_0\id\otimes\id+\sum_{m,l=1}^3\langle\bx_{ml},\overline{\wb}\rangle\sigma_m\otimes\sigma_l,
\end{equation}
Then the corresponding quadratic operator $V$ has the form
\eqref{qqqo1} with constrains $\db=\eb=\gbb=0$. From Theorem
\ref{qqqo2} one immediately gets

\begin{cor}\label{qqqo-22}
  Let the operator $V$ given by \eqref{qqqo1} with $\db=\eb=\gbb=0$. Then
  $V({\mathbf{S}})\subset{\mathbf{S}}$ if
  and only if  the followings hold true
  \begin{enumerate}
\item[(i)] $\|\ab\|=1, \ \|\bb\|=1, \ \|\cb\|=1;$
\item[(ii)]$\|A\|=\|\ab-\bb\|, \ \|\Gamma\|=\|\ab-\cb\|, \
\|B\|=\|\bb-\cb\|;$
\item[(iii)] $\langle\ab,B\rangle+\langle A,\G
\rangle=0, \ \langle\bb,\G\rangle+\langle A,B \rangle=0, \
\langle\cb,A\rangle+\langle B,\G \rangle=0;$

\item[(iv)] $\langle\ab,A\rangle=0, \ \langle\ab,\G\rangle=0, \
\langle\bb,A\rangle=0, \ \langle\bb,B\rangle=0, \
\langle\cb,\G\rangle=0, \ \langle\cb,B\rangle=0$
   \end{enumerate}
where the vectors $\ab,\bb,\cb,A,B,\G$ are given in Theorem
\ref{qqqo2}.
 \end{cor}

Let us consider a symmetric quasi q.q.o. $\D$ with Haar state $\t$,
corresponding to \eqref{qqqo1}. Then according to \eqref{D4-1} the
operator $\D$ has the following
 form
 \begin{eqnarray*}
   \D(x)=w_0\id\o\id&+&a_1w_1\s_1\o\s_1+\frac{\a_1}{2}w_1\s_1\o\s_2+\frac{\g_1}{2}w_1\s_1\o\s_3\\
   &+&\frac{\a_1}{2}w_1\s_2\o\s_1+b_1w_1\s_2\o\s_2+\frac{\b_1}{2}w_1\s_2\o\s_3\\
   &+&\frac{\g_1}{2}w_1\s_3\o\s_1+\frac{\b_1}{2}w_1\s_3\o\s_2+c_1w_1\s_3\o\s_3\\
   &+&a_2w_2\s_1\o\s_1+\frac{\a_2}{2}w_2\s_1\o\s_2+\frac{\g_2}{2}w_2\s_1\o\s_3\\
   &+&\frac{\a_2}{2}w_2\s_2\o\s_1+b_2w_2\s_2\o\s_2+\frac{\b_2}{2}w_2\s_2\o\s_3\\
   &+&\frac{\g_2}{2}w_2\s_3\o\s_1+\frac{\b_2}{2}w_2\s_3\o\s_2+c_2w_2\s_3\o\s_3\\
   &+&a_3w_3\s_1\o\s_1+\frac{\a_3}{2}w_3\s_1\o\s_2+\frac{\g_3}{2}w_3\s_1\o\s_3\\
   &+&\frac{\a_3}{2}w_3\s_2\o\s_1+b_3w_3\s_2\o\s_2+\frac{\b_3}{2}w_3\s_2\o\s_3\\
   &+&\frac{\g_3}{2}w_3\s_3\o\s_1+\frac{\b_3}{2}w_3\s_3\o\s_2+c_3w_3\s_3\o\s_3
 \end{eqnarray*}
Calculating the last one, we obtain
\begin{eqnarray}\label{qqqo9}
  \D(x)=\left(%
\begin{array}{cccc}
  w_0+R & N-iP & N-iP & L-2iM-O \\
  N+iP & w_0-R & L+O & -N+iP \\
  N+iP & L+O & w_0-R & -N+iP \\
  L+2iM-O & -N-iP & -N-iP & w_0+R \\
\end{array}%
\right)
\end{eqnarray}
where
\begin{eqnarray*}
&&L=\langle \ab,\wb\rangle, \ \ \ M=\frac{1}{2}\langle
A,\wb\rangle,\ \ \ N=\frac{1}{2}\langle\G,\wb\rangle, \\[2mm]
&&O=\langle\bb,\wb\rangle, \ \ \ P=\frac{1}{2}\langle B,\wb\rangle,
\ \ \ R=\langle\cb,\wb\rangle.
\end{eqnarray*}

\begin{thm}\label{P-2q}
Let $\D:\bm_2(\bc)\to \bm_2(\bc)\o \bm_2(\bc)$ be a q-pure
symmetric quasi q.q.o. with Haar state $\t$. Then $\D$ is not
positive.
\end{thm}

\begin{proof}
  Let us prove from the contrary. Assume that $\D$ is positive. This
  means that the matrix given by \eqref{qqqo9} should be positive,
  whenever $x$ is positive. The positivity of $x$ yields that $w_0, w_1, w_2,
  w_3$ are real numbers. In what follows, without loss of
  generality, we may assume that $w_0=1$, and therefore
  $\|\wb\|\leq1$. It is known that the positivity of the matrix
  $\D(x)$ is equivalent to the positivity of its eigenvalues, and
  it should be positive for any values of $\|\wb\|\leq1$.

We note that the $q$-purity of $\D$ implies that the conditions
(i)-(iv) of Corollary \ref{qqqo-22} are satisfied.

Let us take $x=\id+\ab\s$, then from \eqref{qqqo9} one gets
  \begin{eqnarray*}
    \D(x)=\left(%
\begin{array}{cccc}
  1+\langle\cb,\ab\rangle & -\frac{i}{2}\langle B,\ab\rangle & -\frac{i}{2}\langle B,\ab\rangle & 1-\langle\bb,\ab\rangle  \\
  \frac{i}{2}\langle B,\ab\rangle & 1-\langle\cb,\ab\rangle & 1+\langle\bb,\ab\rangle  & \frac{i}{2}\langle B,\ab\rangle \\
  \frac{i}{2}\langle B,\ab\rangle & 1+\langle\bb,\ab\rangle  & 1-\langle\cb,\ab\rangle & \frac{i}{2}\langle B,\ab\rangle \\
  1-\langle\bb,\ab\rangle  & -\frac{i}{2}\langle B,\ab\rangle & -\frac{i}{2}\langle B,\ab\rangle & 1+\langle\cb,\ab\rangle \\
\end{array}%
\right).
  \end{eqnarray*}
  A simple algebra shows us that all eigenvalues of $\D(x)$ can be
  written as follows
  \begin{eqnarray*}
    &&\l_1=-\langle\cb,\ab\rangle-\langle\bb,\ab\rangle\\
    &&\l_2=\langle\cb,\ab\rangle+\langle\bb,\ab\rangle\\
    &&\l_3=2+\sqrt{\langle\bb,\ab\rangle^2-2\langle\cb,\ab\rangle\langle\bb,\ab\rangle+\langle\cb,\ab\rangle^2+\langle
    B,\ab\rangle^2}\\
    &&\l_4=2-\sqrt{\langle\bb,\ab\rangle^2-2\langle\cb,\ab\rangle\langle\bb,\ab\rangle+\langle\cb,\ab\rangle^2+\langle
    B,\ab\rangle^2}.
  \end{eqnarray*}
  Now using (ii) of Corollary \ref{qqqo-22} we rewrite $\l_1, \l_2, \l_3, \l_4$ as
  follows
  \begin{eqnarray*}
    &&\l_1=-2+\frac{\|\G\|^2}{2}+\frac{\|A\|^2}{2}\\
    &&\l_2=2-\frac{\|\G\|^2}{2}-\frac{\|A\|^2}{2}\\
    &&\l_3=2+\frac{1}{2}\sqrt{\|A\|^4-2\|\G\|^2\|A\|^2+\|\G\|^4+\langle
    B,\ab\rangle^2}\\
    &&\l_4=2-\frac{1}{2}\sqrt{\|A\|^4-2\|\G\|^2\|A\|^2+\|\G\|^4+\langle
    B,\ab\rangle^2}.
  \end{eqnarray*}
  Knowing $\l_1\geq0, \l_2\geq0$ we have
\begin{eqnarray*}
  \|A\|^2+\|\G\|^2=4.
  \end{eqnarray*}

By considering elements $x=\id+\bb\s$, $x=\id+\cb\s$,
respectively, and using the similar argument one finds
  \begin{eqnarray*}
    \|B\|^2+\|A\|^2=4 \ \
    \|\G\|^2+\|B\|^2=4
  \end{eqnarray*} Therefore, we
  conclude that
  \begin{eqnarray*}
    \|A\|^2=2, \ \|B\|^2=2, \ \|\G\|^2=2.
  \end{eqnarray*}
  Hence, again taking into account (ii) of Corollary \ref{qqqo-22} we find that
  \begin{eqnarray*}
    \langle\ab,\bb\rangle=0, \ \langle\ab,\cb\rangle=0, \
    \langle\bb,\cb\rangle=0.
  \end{eqnarray*}
  This means that the vectors $\ab, \bb, \cb$ are linearly
  independent. Therefore, one can write
  \begin{eqnarray*}
    A=\eta_1\ab+\mu_1\bb+\tau_1\cb\\
    B=\eta_2\ab+\mu_2\bb+\tau_2\cb\\
    \G=\eta_3\ab+\mu_3\bb+\tau_3\cb
  \end{eqnarray*}
  where $\eta_i^2+\mu_i^2+\tau_i^2=2, \ i=\overline{1,3}.$

  From (iv) of Corollary \ref{qqqo-22} we find that
  \begin{eqnarray*}
    \eta_1=0, \ \eta_3=0, \ \mu_1=0, \ \mu_2=0, \ \tau_2=0, \ \tau_3=0.
  \end{eqnarray*}
  This implies that
  \begin{eqnarray*}
    A=\tau_1\cb, \ \ B=\eta_2\ab, \ \ \G=\mu_3\bb.
  \end{eqnarray*}
  Hence, from (iii) of Corollary \ref{qqqo-22} we have
  \begin{eqnarray*}
  \langle A,B\rangle+\langle\bb,\G\rangle=0 \ \ \Rightarrow \ \
  \mu_3=0\\
  \end{eqnarray*}
  which contradicts to $\mu_3\neq0$. This completes the proof.
\end{proof}

This theorem implies that $q$-pure symmetric quasi q.q.o. with Haar
state can not be q.q.o. Moreover, if one has pure quasi q.q.o., then
it cannot be positive. As we have seen in the previous section a
quasi q.q.o. with only "linear" term can be positive. But the last
theorem shows the difference between Theorem \ref{P-q}. Namely, if
one considers a quadratic operator $V$ which is linear (this
corresponds to the case of Theorem \ref{P-q}), then $q$-pure quasi
q.q.o. is positive. But Theorem \ref{P-2q} implies a different kind
of statement, i.e. if $V$ contains a nonlinear term, i.e. quadratic
term, then the $q$-purity of $\D$ does not imply its positivity.

\section{On dynamics of q-pure quasi quantum quadratic operator.}

In this section we are going to make some remarks on dynamics of
$q$-pure quasi q.q.o.

Let $\D$ be a $q$-pure quasi q.q.o. By $V$ we denote the
corresponding quadratic operator. Now we want to study the dynamics
of $V$.

\begin{prop}\label{qqqoP}
  Let $V$ be a quadratic operator corresponding to $q$-pure quasi q.q.o. with
  Haar state $\tau$. Then for any $\fb\in\mathbf{D}\setminus\mathbf{S}$ one
  has
  \begin{eqnarray*}
    \lim_{n\rightarrow\infty}{V^n(\fb)}=0.
  \end{eqnarray*}
\end{prop}
\begin{proof}
  Let $\fb\in\mathbf{D}\setminus\mathbf{S}$ then one can see
  $\|\fb\|<1.$ Denote $\mathbf{g}=\frac{\fb}{\|\fb\|}$ then
  $\mathbf{g}\in\mathbf{S}$. Therefore using purity of $\D$ we
  conclude $V(\mathbf{g})\in\mathbf{S}.$ This means
  \begin{eqnarray*}
    1=\bigg\|V\bigg(\frac{\fb}{\|\fb\|}\bigg)\bigg\|=\frac{1}{\|\fb\|^2}\|V(\fb)\|.
  \end{eqnarray*}
  So
  \begin{eqnarray*}
    \|V(\fb)\|=\|\fb\|^2.
  \end{eqnarray*}
  Hence, we find
  \begin{eqnarray*}
    \|V^n(\fb)\|=\|\fb\|^{2^n}
  \end{eqnarray*}
  which implies $V^n(\fb)\rightarrow0$ as $n\rightarrow\infty$.
\end{proof}

\begin{cor}
  Let $V$ be as in Proposition \ref{qqqoP}, then any nonzero fixed
  point (if it exists) belongs to $\mathbf{S}$. Moreover,
  $(0,0,0)$ is unique fixed point in $\mathbf{D}\setminus\mathbf{S}$
\end{cor}

Now to investigate dynamics of $V$ it remains to study it on
$\mathbf{S}$. Next examples show how the dynamics could be
unpredictable on $\mathbf{S}$.

{\bf 1.} Let us study the dynamics of the given operator $V_0$
given by \eqref{ex1}. We consider several cases.

Now assume $\fb\in \mathbf{S}$. Suppose that $f_1=0$. Then we have
\begin{eqnarray*}
  V_0(\fb)=(0,-1,0).
\end{eqnarray*}
Hence, $V^n_0(\fb)=(0,-1,0)$, for every $n\in \mathbb{N}$.

Suppose that $f_2=0$. Then
\begin{eqnarray*}
  V_0(f_1,0,f_3)=(0,f_1^2-f_3^2,2f_1f_3).
\end{eqnarray*}
So we have $V^k_0(\fb)\rightarrow(0,-1,0)$ as
$n\rightarrow\infty$.

Suppose that $f_3=0$. Then
\begin{eqnarray*}
  V_0(f_1,f_2,0)=(2f_1f_2,f_1^2-f_2^2,0)=\bigg(\pm 2f_1\sqrt{1-f_1^2},2f_1^2-1,0\bigg).
\end{eqnarray*}
To investigate the dynamics of $V_0$, let us consider the
following function
\begin{eqnarray*}
  g(x)=2x\sqrt{1-x^2}, \ |x|\leq1.
\end{eqnarray*}
For us it is enough to study the dynamics of $g(x)$. It is clear
that
\begin{eqnarray*}
  g[0,1]\subset[0,1], \ g[-1,0]\subset[-1,0].
\end{eqnarray*}
Since the function is odd it is sufficient to study the dynamics
of $g$ on $[0,1]$. Denote $h(x)=\sqrt{x}$. One can see that
\begin{eqnarray*}
  h^{-1}\big(g\big(h(x)\big)\big)=4x(1-x).
\end{eqnarray*}
This means $g(x)$ and $\ell(x)=4x(1-x)$ are conjugate on [0,1]. It
is known that the function $\ell(x)$ is the logistic function which
is chaotic. Hence, $g(x)$ is also chaotic. From this we conclude
that the behavior of $V_0$ on $\mathbf{S}$ with $f_3=0$ is chaotic.
Note that similar kind of dynamical system has been investigated in
\cite{BGLL,Mal2,Sz}.

{\bf 2.} Let
\begin{eqnarray*}
  \D_1(x)=w_0\id\o\id+\langle\mathbf{t},\wb\rangle(\s_1\o\s_1+\s_2\o\s_2+\s_3\o\s_3),
\end{eqnarray*}
then the corresponding quadratic operator has the following form
\begin{eqnarray*}
  V_1(f_1,f_2,f_3)=\left\{\begin{array}{lll}
    t_1\big(f_1^2+f_2^2+f_3^2\big) \\
    t_2\big(f_1^2+f_2^2+f_3^2\big) \\
    t_3\big(f_1^2+f_2^2+f_3^2\big) \\
  \end{array}
  \right.
\end{eqnarray*}
where $\|\mathbf{t}\|=1, \ \mathbf{t}=(t_1,t_2,t_3).$

One can see $V_1$ has only two fixed points which are $(0,0,0), \
(t_1,t_2,t_3)$. It is easy to see that
$V(\mathbf{S})=\{\mathbf{t}\}$, so $\D_1$ is $q$-pure quasi q.q.o.
Therefore, we conclude that
\begin{eqnarray*}
  \lim_{n\rightarrow\infty}{V_1^n(\fb)}=\left\{\begin{array}{lll}
   \mathbf{t}, \ \ \fb\in\mathbf{S} \\
   \mathbf{0}, \ \ \fb\in\mathbf{D}\setminus\mathbf{S}. \\
  \end{array}
  \right.
\end{eqnarray*}

\section{Conclusion}

In the present paper we studied quasi quantum quadratic operators
(q.q.o) acting on the algebra of $2\times 2$ matrices $\bm_2(\bc)$.
We have introduced a weaker condition, called $q$-purity, than
purity of the channel. To study $q$-pure channels, we have described
all trace-preserving quasi q.q.o. acting on $\bm_2(\bc)$, which
allowed us to describe all $q$-pure quadratic operators. Then we
prove that if a trace-preserving symmetric quasi q.q.o. such that
the corresponding quadratic operator is linear, then its $q$-purity
implies its positivity. Moreover, if a symmetric quasi q.q.o. has a
Haar state $\t$, then its corresponding quadratic operator is
nonlinear, and it is proved that such $q$-pure symmetric quasi
q.q.o. cannot be positive. Note that there are nontrivial q.q.o.
such that their corresponding quadratic operators are nonlinear
\cite{MA}. We think that such a result will allow to check whether a
given mapping from $\bm_2(\bc)$ to $\bm_2(\bc)\o\bm_2(\bc)$ is pure
or not. On the other hand, our study is related to construction of
pure quantm nonlinear channels. We should stress that nonlinear
channels appear in many branches of quantum information (see for
example \cite{BB,IOS,OV}). Moreover, one also established that
nonlinear dynamics associated with quasi pure q.q.o. may have
differen kind of dynamics, i.e. it may behave chaotically or
trivially, respectively.

\section*{Acknowledgement} The authors
acknowledges the MOHE grant FRGS11-022-0170 and the Junior Associate
scheme of the Abdus Salam International Centre for Theoretical
Physics, Trieste, Italy. Finally, the author also would like to
thank to an anonymous referee whose useful suggestions and comments
improve the content of the paper.

\end{document}